\documentclass[10pt,a4paper]{article}
\usepackage[T1]{fontenc}
\usepackage[utf8]{inputenc}
\pdfoutput=1
\usepackage{lmodern}
\usepackage{textcomp}
\usepackage{microtype}
\usepackage[DIV=10,BCOR=5mm,headinclude=false,footinclude=false]{typearea}
\usepackage[font=small,labelfont=bf]{caption}
\usepackage{titlesec}
\usepackage{tocloft}
\usepackage{amsmath,amssymb}
\usepackage{enumerate} 
\usepackage{amsfonts}
\usepackage{graphicx}
\usepackage{amsthm}
\usepackage{yfonts}
\usepackage{tikz-cd}
\usepackage{amssymb}
\usepackage{mathrsfs}
\usepackage{bbm}
\usepackage[parfill]{parskip}
\usepackage{hyperref}
\usepackage[nameinlink]{cleveref}
\usepackage{MnSymbol }
\usepackage{verbatim}
\usepackage{array}
\usepackage{tensor}
\usepackage{enumitem}

\newtheorem{theorem}{Theorem}

\newtheorem{corollary}{Corollary}[section]
\newtheorem{lemma}{Lemma}[section]
\newtheorem{prop}{Proposition}[section]

\theoremstyle{remark}
\newtheorem{rem}{Remark}[section]
\newcommand{\Ff}{\mathrm{F}}
\newcommand{\ra}{\rightarrow}
\newcommand{\ql}{{\overline{\mathbb{Q}}_\ell}}
\newcommand{\fl}{{\overline{\mathbb{F}}_\ell}}

\newcommand{\sra}{\twoheadrightarrow}
\newcommand{\hra}{\hookrightarrow}
\newcommand{\rep}{\mathfrak{Rep}}
\newcommand{\id}{\mathrm{Ind}}
\newcommand{\as}{{\lvert-\rvert}}

\newcommand{\abs}{\nu}
\newcommand{\NN}{\mathbb{N}}
\renewcommand*{\det}{\qopname\relax o{det}}
\newcommand{\ZZ}{\mathbb{Z}}

\newcommand{\irr}{\mathfrak{Irr}}
\newcommand{\ain}[3]{#1\in\{#2,\ldots,#3\}}
\newcommand{\ho}{\mathrm{Hom}}
\newcommand{\soc}{\mathrm{soc}}

\newcommand{\De}{\Delta}
\newcommand{\Z}{\mathrm{Z}}

\newcommand{\cc}{^{\mathfrak{c}}}

\newcommand\restr[2]{{\left.\kern-\nulldelimiterspace #1\vphantom{\big|}\right|_{#2}}}

\renewcommand{\abs}[2]{\nu_{#1}^{\frac{#2}{2}}}

\newcommand{\cus}{\mathfrak{C}}

\newcommand{\im}{\mathrm{Im}}

\newcommand{\Ms}{\mathcal{M}\mathrm{ult}}
\newcommand{\fm}{\mathfrak{m}}
\newcommand{\fn}{\mathfrak{n}}

\newcommand{\fd}{\mathfrak{d}}
\newcommand{\rs}{r^\Delta}

\newcommand{\ed}{\mathrm{End}}
\newcommand{\pa}{\mathrm{P}ar}
\newcommand{\ad}{\mathrm{Ad}}
\newcommand{\pp}{\mathfrak{PP}}
\newcommand{\Md}{\mathrm{Mod}}
\newcommand{\md}{\mathrm{mod}}
\newcommand{\rst}{\mathrm{rest}}
\newcommand{\mt}{\mathrm{Mat}}
\newcommand{\fo}{\mathbbm{1}}
\newcommand{\D}{\mathcal{R}ed}

\newcommand{\oltimes}{\overline{\times}}
\newcommand{\sk}{\mathrm{sk}}
\newcommand{\ol}{\overline}

\newcommand{\scu}{\mathfrak{SC}}
\newcommand{\bs}{\backslash}
\newcommand{\len}{\mathrm{len}}
\newcommand{\rk}{\mathrm{rk}}
\newcommand{\CC}{\mathbb{C}}

\newcommand{\otimeslr}{\tensor*[ _R]{\otimes}{_L}}
\date{}
\begin{document}
\title{The $\ell$-modular local theta correspondence in type II and partial permutations}
\author{Johannes Droschl}
\maketitle
\begin{abstract}
    In this paper we compute the multiplicities appearing in the $\fl$-modular theta correspondence in type II over a non-archimedean field $\Ff$, where $\ell$ is a prime not dividing the residue cardinality of $\Ff$.
    Unlike for representations with complex coefficients, highly non-trivial multiplicities can emerge. We show that these multiplicities are precisely governed by the action of symmetric groups on the set of partial permutations, and the $\fl$-representation of symmetric groups these give rise to. The problem is thus reduced to certain branching problems in the modular representation theory of symmetric groups.
    In particular, if $d$ is the order of the residue cardinality of $\Ff$ in $\fl$, and the rank of the involved general linear groups is bounded above by $d\ell$, the behavior of the theta correspondence can be predicted via explicit algorithms coming from Pieri's formula.
\end{abstract}
\section{Introduction}
Let $\Ff$ be a non-archimedean local field with residue cardinality $q$. We fix a prime $\ell$ such that $q$ is invertible in $\fl$ and denote by $d$ the order of $q$ in $\fl$.
In this paper we investigate the $\ell$-modular theta correspondence in type II. Originally developed in \cite{How79}, \cite{How89}, see also \cite{Wal89}, the theta correspondence is one of the main tools in modern local and global representation theory giving rise to lifts of representations of one algebraic group $G$ over $\Ff$ to another one $H$. One of its first applications was the construction of counter-examples to the naive Ramanujan-Petersson conjecture, \cite{HowPia83}, more recent cases being specific constructions of the local Langlands correspondence, \cite{GanTak11}, and the Gan-Gross-Prasad conjectures, \cite{GanIch16}.
The best known instance of the theta correspondence occurs in the case where $G$ is a symplectic group and $H$ is an orthogonal group, or 
\emph{vice versa}, which is known as the theta correspondence in type I.
In this paper, we will however focus on the case of type II, \emph{i.e.} where $G$ and $H$ are two general linear groups. The reason being that the goal of the paper is to understand $\ell$-modular versions of the theta correspondence, since this is the setting in which the $\ell$-modular representation theory is best understood, thanks to the work of \cite{MinSec14} and \cite{Vig96}. We will see that many of the known properties in the case of complex representations, originally proven in \cite{Min08}, break down in interesting ways. Partial results towards modular versions of the theta correspondence have already been obtained in the works of \cite{MosTri}, \cite{Min08}, and \cite{Min12}.
The main motivation for considering the $\ell$-modular versions is that some of the deepest and most surprising arithmetic properties of automorphic forms manifest themselves as certain congruences modulo some prime $\ell$. Those congruences are then directly related to the $\ell$-modular representation theory of the corresponding local representations. Furthermore, the deformation-theory of (local) Galois-representations suggests to us that it might be beneficial to study representations not in a fixed field of coefficients, but rather to vary the coefficients in families, see for example \cite{EmeHel14}, \cite{Hel20}, and \cite{HelMos18}.
An important example being 
representations over $\ol{\ZZ}_\ell$, which specialize to representations over $\ql$ and $\fl$. In particular, one is interested in constructing invariants of representations that behave well in families. So is for example the $\gamma$-factor of representations of the general linear group, see \cite{Min12}, such a well-behaved object, but the standard Godement-Jacquet $L$-factor introduced in the same paper is not. Both these objects are intimately tied to the theta correspondence in type II, as we will try to highlight throughout the remaining introduction.

Let us recall the main players of our story. We let $G_n$ be the $\Ff$-points of the general linear group and $\mt_{n,m}$ the space of $n\times m$-matrices with entries in $\Ff$, on which $G_n\times G_m$ acts by left- and right-translation. For an algebraically closed field $K$, we let $S_K(X)$ be the space of compactly supported and locally constant functions $f\colon X\ra K$. 
For irreducible, smooth $K$-representations $\pi$ of $G_n$ and $\pi'$ of $G_m$ one considers in the theta correspondence in type II the Hom-space
\[\ho(S_K(\mt_{n,m}),\pi\otimes\pi').\]
Howe's duality conjecture, which was originally postulated in \cite{How79}, then asserts the following.
\begin{theorem}[{\cite{Min08}}]
    Assume that $\mathrm{char}(K)=0$, $n\le m$, and let $\pi$ be an irreducible and smooth $K$-representation of $G_n$. Then there exists a unique (up to isomorphism) irreducible, smooth $K$-representation $\pi'$ of $G_m$ with
    $\ho(S_K(\mt_{n,m}),\pi\otimes\pi')\neq 0$. In this case
    \[\dim_K\ho(S_K(\mt_{n,m}),\pi\otimes\pi')=1\]
    and if $n=m$, we have $\pi'\cong \pi^\lor$.
\end{theorem}
In particular, if $K=\CC$ and $n=m$, the morphism
\[S_\CC(\mt_{n,n})\ra\pi\otimes \pi^\lor\] is constructed as follows. First, one constructs for generic $s\in \CC$ a morphism
\[f_s\colon S_\CC(\mt_{n,n})\ra\pi\as^s\otimes \pi^\lor\as^{-s}\]
and then analytically continues $f_s$ to the whole complex plane. 
After accounting for a possible pole at $s=0$, one can evaluate $f_s$ at $s=0$ and the order of the pole is by definition the order of the pole of the Godement-Jacquet $L$-function at $s=\frac{1-n}{2}$.

We now come to the main contents of the paper and set from now on $K=\fl$. Before stating our main theorem, we need to recall some concepts from the modular representation-theory of the symmetric group $S_n$. We recall firstly that the irreducible representations of $S_n$ are parametrized by $\ell$-restricted partitions of $n$, \emph{i.e.} partitions in which every entry is repeated at most $\ell-1$-times. We denote the set of $\ell$-restricted partitions by $\pa_n^\ell$, $\pa^\ell=\bigcup_{n\in\NN}\pa_n^\ell$, and the representation corresponding to $\lambda\in\pa_n^\ell$ by $D_\lambda$. Our convention will be that $\pa_0^\ell$ consists of the $0$-partition $(0)$ and the empty partition $\emptyset$. 

Next, we consider for $n,m\in\NN$ the set \[\pp_{n,m}\coloneq \{(U,V,f):U\subseteq \{1,\ldots,n\},V\subseteq \{1,\ldots,m\}, f\colon V\ra U\text{ bijective}\}\] of \emph{partial permutations}, on which
the groups $S_n$ and $S_m$ act from the left and the right. For $\lambda\in\pa_n^\ell$ and $\mu\in\pa_m^\ell$ we let
\[d_{\lambda,\mu}\coloneq\dim_\fl\ho(\fl[\pp_{n,m}],D_\lambda {\otimeslr} D_\mu).\]
If $\lambda$ or $\mu$ is the empty partition, we set $d_{\lambda,\mu}=0$, and if both are non-empty and one of them is $0$, we set $d_{\lambda,\mu}=1$.
In \Cref{S:symmetric}, we discuss how to compute $d_{\lambda,\mu}$ if $n,m<\ell$ and remark how in general it is related to branching problems in the theory of symmetric groups.

Let $\irr_n$ be the set of irreducible, smooth $\fl$-representations of $G_n$. In \Cref{S:final} we define via representation-theoretic means a map, called the \emph{skeleton},
\[\sk\colon \irr_n\times\irr_m\ra\pa^\ell\times\pa^\ell.\]
Before we can remark on some properties of this map, we need to introduce the following notation.
Let $r\in \NN$, recall that $d$ is the order of the residue cardinality in $\fl$, and denote by $\Pi_r\coloneq \id_{P_{d,\ldots,d}}^{G_{rd}}(\fo_d\otimes\ldots\otimes\fo_d)$ the normalized parabolic induction of the trivial representation of the Levi-subgroup of the standard parabolic subgroup of $G_{rd}$ corresponding to the partition $(d,\ldots,d)$.
Then the Zelevinsky-classification of irreducible representations of \cite{MinSec14} allows us to associate to any irreducible subquotient $\pi$ of $\Pi_r$ a partition $\lambda$ of $r$, in which case we write $\pi=\Z(\lambda)$. In \Cref{S:move} we prove that $\Z(\lambda)$ is a subrepresentation of the induced representation $\Pi_r$ if and only if $\lambda$ is  $\ell$-restricted. The heart of this section is the following theorem, which we prove using the intertwining operators of \cite{Dat05}.
\begin{theorem}
    Let $r\in\NN$. Then there exists an explicit isomorphism of algebras
    \[T\colon \fl[S_r]\ra\ed(\Pi_r),\]
    where multiplication on $\ed(\Pi_r)$ is given by composition. 
\end{theorem}
Although we do not treat this topic in this paper, we believe that the above theorem can be used to give an alternative proof of the classification in \cite{MinSec14}, completely bypassing the theory of Iwahori-Hecke algebras.

We now come to the promised properties of $\sk$.
\begin{enumerate}
    \item The image of $\sk(\irr_n,\irr_m)$ consists of partitions of at most $\frac{n}{d}$ in the first, and at most $\frac{m}{d}$ in the second component.
    \item If $\lambda,\mu\in \pa^\ell$, then $\sk(\Z(\lambda),\Z(\mu))=(\lambda,\mu)$.
    \item Let $\pi_1,\pi_2\in\irr_m$ be such that the first components of $\sk(\pi,\pi_1),\, \sk(\pi,\pi_2)$ are different from $\emptyset$. Then the first components agree, and we will denote them by $\sk_m(\pi)$. If no $\pi'\in\irr_m$ exists such that the first component of $\sk(\pi,\pi')$ is different from $\emptyset$, we set $\sk_m(\pi)=\emptyset$. 
    \item If $m=m'\mod d$, $\sk_m(\pi)=\sk_{m'}(\pi)$. 
\end{enumerate}
Furthermore, if $d=1,$ (2) allows us to compute $\sk$ in terms of the Zelevinsky-classification, and hence by the Mullineux correspondence, in terms of the Langlands classification.
At the moment we cannot offer a general algorithm to compute $\sk_m$ in terms of the Zelevinsky-classification in the case $d>1$.

In \Cref{S:finalbig} we then prove the main theorem of the paper.
\begin{theorem}
    Let $\pi\in \irr_n$, $\pi'\in \irr_m$. Then
    \[\dim_\fl\ho(S_\fl(\mt_{n,m}),\pi\otimes\pi')=d_{\sk(\pi,\pi')}.\]
\end{theorem}
In particular,  if $\ell d>\max(n,m)$ one can explicitly compute the multiplicities using Pieri's formula.
The main techniques we use are the ones developed in \cite{Dro25II}, together with some Fourier-theoretic techniques, which were inspired by the original proof of Howe duality in type II in \cite{Min08}.

We return now for a moment to the $L$-functions mentioned at the beginning of the introduction and end this introduction on a speculative note. It is well known that they do not behave well in families, and in fact it is the $\gamma$-factors of Godement-Jacquet $L$-functions that are better behaved. For example, assume $n=dr$. Then
the $L$-function of the trivial $\ql$-representation $\widetilde{\fo_n}$, $L(T,\widetilde{\fo_n})$ has simple poles at $T=q^{\frac{-n-1+2i}{2}},\, \ain{i}{1}{n}$ and the coefficient-wise reduction $\mod\ell$, $r_\ell(L(T,\widetilde{\fo_n}))$, has a pole of order $r$ at $T=q^{\frac{1-n}{2}}$.
But on the other hand, if $\fo_n$ is the trivial $\fl$-representation of $G_n$, the $L$-function $L(T,\fo_n)$ has a simple pole at $T=q^{\frac{1-n}{2}}$ if $d>1$ and no pole if $d=1$, see \cite{Dro24}.
It seems that the poles of $L(T,\widetilde{\fo_n})$ wandered off into extra morphisms in $\ho(S_\fl(\mt_{n,n}),\fo_n\otimes\fo_n)$.
Indeed, the results of \cite[Theorem 3]{Dro25II}, together with \Cref{L:int2}, hint that the following modification of $L(T,\pi)$ might be better behaved.
\[L^\theta(T,\pi)\coloneq \prod_{i\in \ZZ/(d\ZZ)}(1-q^{-\frac{1-n+2i}{2}}T)^{-d_i(\pi)}L(T,\pi),\] where\[ d_i(\pi)\coloneq \dim_\fl\ho(S(\mt_{n,n}),\as^{i}\pi\otimes\as^{-i}\pi^\lor)-1.\]
In particular, it follows that $L^\theta(T,\pi)$ satisfies the same functional equation as $L(T,\pi)$ by \Cref{T:reduc2} and if $n<d$ it agrees with the standard Godement-Jacquet $L$-factor.
We hope to explore this more in depth in future work. 
\subsection*{Acknowledgements}
I am grateful to Hengfei Lu for his valuable feedback and for allowing me to double-check my results against previous computations conducted by him. I would also like to thank Corina Ciobotaru and Alberto M{\'i}nguez for the insightful discussions on earlier versions of this manuscript. The author was supported by a research grant (VIL53023) from VILLUM FONDEN.
\section{The setup}
\numberwithin{theorem}{section}
 We fix for the rest of the paper a prime $\ell$ such that $q$ is invertible in $\fl$ and denote the order of $q$ in $\fl$ by $d$. We also normalize the choice of an absolute value $\lvert-\lvert$ on $\Ff$ by demanding that $\lvert\varpi\rvert=q^{-1}$, where $\varpi$ is a uniformizer of $\Ff$.
\subsection{Partitions}
Let us fix our conventions regarding partitions and compositions. For us a composition $\lambda$ of a natural number $n$ will be an ordered tuple $(\lambda_1,\ldots,\lambda_k)$ with entries some positive integers such that 
\[\sum_{i=1}^k\lambda_i=n.\] The opposite composition $\ol{\lambda}$ is defined as $(\lambda_k,\ldots,\lambda_1)$. For $k\in \ZZ_{>0}$, we denote by $k\lambda$ the composition of $nk$ given by entrywise multiplication.
If $\mu=(\mu_1,\ldots,\mu_l)$ is a second composition of $n$, we write $\mu\ge \lambda$ if for all $\ain{i}{1}{\min(k,l)}$ 
\[\sum_{j=1}^i\mu_j\ge \sum_{j=1}^i\lambda_j.\]

If $\lambda_1\ge\ldots\ge\lambda_k$, we refer to the composition $\lambda$ as a partition. We denote by $\pa_n$ the set of partitions of $n$ and $\pa=\bigcup_{n\in\NN}\pa_n$. For our purposes we let $\pa_0$ be a set with two elements which we denote by $0$ and $\emptyset$.
\subsection{Partial permutations and the symmetric group}\label{S:symmetric}
We start by investigating what will turn out to be the core of our considerations, namely the symmetric group, its $\fl$-representations and its action on the set of partial permutations.
This will be a brief summary of the results relevant to us. For a reader interested in a more exhaustive treatment, we refer for example to \cite{GorKer81}. 

Let $S_n$ be the $n$-th symmetric group. Its generators are given by elementary transpositions and are denoted by $s_1,\ldots,s_{n-1}$. We let $\len\colon S_n\ra \NN$ denote the length function sending $w=s_{i_1}\ldots s_{i_k}\,,s_{i_j}\in\{s_1,\ldots,s_{n-1}\}$ to $k$ if $s_{i_1}\ldots s_{i_k}$ is a reduced expression of $w$.

For $\alpha$ a composition of $n$, we consider the standard subgroup $S_\alpha=S_{\alpha_1}\times\ldots\times S_{\alpha_k}$ and $\fl[S_\alpha]$ the associated algebra.
We let $\Md_L^\alpha$ and $\Md_R^\alpha$ be the category of finite dimensional left- respectively right-modules of $\fl[S_\alpha]$ and denote the induction functor
\[\id_\alpha\colon \Md_*^\alpha\ra\Md_*^n,\,*\in\{L,R\},\]
with right-adjoint the restriction functor.
Similarly, if $\beta$ is a second composition of $m$, we let $\Md^{\alpha,\beta}$ be the category of finite dimensional left-right modules of $\fl[S_\alpha]\otimes \fl[S_\beta]$ and denote the induction functor
\[\id_{\alpha,\beta}\colon \Md^{\alpha,\beta}\ra\Md^{n,m},\]
and again with right-adjoint the restriction functor.
Finally, we denote as usual the tensor-products
\[\otimeslr\colon \Md_L^\alpha\times \Md_R^\beta\ra \Md^{\alpha,\beta}.\]
If $(k,n-k)$ is a composition of $n$, we note that there is an isomorphism of $S_{k,n-k}$-modules $\fl[S_n]\cong \binom{n}{k}\cdot \fl[S_{k,n-k}]$.
The isomorphism is given as follows. Any $w\in S_n$ can be written uniquely as $w=vu$, where $u\in S_{k,n-k}$ and $v$ is the unique element of minimal length in $wS_{k,n-k}$. Next we number the $\binom{n}{k}$ $S_{k,n-k}$-cosets and send an element $w\in w_iS_{k,n-k}$ to the element $u$ in the $i$-th copy of $S_{k,n-k}$.

Let \[\pp_{n,m}\coloneq \{(U,V,f):U\subseteq \{1,\ldots,n\},V\subseteq \{1,\ldots,m\}, f\colon V\ra U\text{ bijective}\}\] be the set of \emph{partial permutations}.
The groups $S_n$ and $S_m$ act from the left and the right and
the orbits are given precisely by $\pp_{n,m}^r,\, \ain{r}{0}{\min(n,m)}$, where \[\pp_{n,m}^r=\{(U,V,f):\#U=\#V=r\}.\]

We thus obtain a left-$\fl[S_n]$, right-$\fl[S_m]$ module \[\fl[\pp_{n,m}]=\bigoplus_{r=0}^{\min(n,m)}\fl[\pp_{n,m}^r].\]
Moreover, it is not hard to see that
\[\fl[\pp_{n,m}^r]\cong \id_{(r,n-r),(m-r,r)}(\fo_{n-r}\otimeslr  \fl[S_r]{}_{L}\otimes_R\fo_{m-r}).\]

We recall that irreducible representations of $S_n$ are parametrized by $\ell$-restricted partitions $\lambda$ of $n$, \emph{i.e.}, partitions $(\lambda_1,\ldots,\lambda_k)$ such that $\lambda_i$ appears at most with multiplicity $\ell-1$. We denote this set by $\pa_n^\ell$ and $\pa^\ell=\bigcup_{n\in\NN}\pa_n^\ell$.
We denote the representation associated to $\lambda\in\pa^\ell$ by $D_\lambda$. The trivial representation corresponds to the trivial partition $(n)$ for any $n$. To $\lambda$, we can also associate its Specht-module $Sp_\lambda$, which admits $D_\lambda$ as a unique quotient, and is hence generated by a single element. Note that if $n<\ell$, $Sp_\lambda=D_\lambda$. Furthermore, we note that $Sp_\lambda$ is a subrepresentation of \[M_\lambda\coloneq \id_\lambda (\fo_{\lambda_1}\otimes\ldots\otimes \fo_{\lambda_k}).\] Finally, $D_\mu$ appears in the composition series of $M_\lambda$ if and only if $\mu\le \lambda$ and $D_\lambda$ appears with multiplicity $1$. We denote the left- and right-module of $\fl[S_n]$ corresponding to a representation with the same letter.

Next, for $M\in \Md_L^\alpha$ we denote by $\widetilde{M}$ the right-module obtained by twisting the action of $S_\alpha$ by $w\mapsto w^{-1}$ and vice versa. Note that we have $\widetilde{D_\lambda}\cong D_\lambda$ and $\widetilde{Sp_\lambda}\cong Sp_\lambda$.

In the latter parts of the paper the following multiplicities will play a crucial role. We assume without loss of generality that $n\le m$.
Let $\lambda$ be a partition of $n$ and $\mu$ a partition of $m$. We then set for $\ain{r}{0}{n}$
\[d_{\lambda,\mu}^r=\dim_\fl\ho(\fl[\pp_{n,m}^r],D_\lambda\otimeslr D_\mu),\]
\[d_{\lambda,\mu}=\sum_{r=0}^{n}d_{\lambda,\mu}^r=\dim_\fl\ho(\fl[\pp_{n,m}],D_\lambda\otimeslr D_\mu).\]
If $\lambda$ or $\mu$ is equal to $0$ and the other does not equal $\emptyset$, we set $d_{\lambda,\mu}\coloneq 1$, and if one of them is $\emptyset$, then $d_{\lambda,\mu}\coloneq 0$.
As an example, note that if $\lambda=\mu=(n)$,
\[d_{\lambda,\mu}^r=1,\, d_{\lambda,\mu}=n+1.\]

Let us note that if $n\ge \ell$, the computation of these multiplicities is an open problem and only in the case $r=n-1$ and $n=m$ and $r=0$ and $n=1=m$ an explicit answer is known, see \cite{KleI}, \cite{KleII}, \cite{KleIII}, \cite{KleIV}. 
In particular, if $n+1=m$ then $d_{\lambda,\mu}\le 1$. We can obtain explicit upper bounds if $\ell$ is large enough, that is in the semi-simple range, \emph{i.e.} $m<\ell$. Then it is possible to compute these multiplicities using the following two ingredients.
   Let $\lambda$ be a partition of $n-k\ge 0$ with entries $\lambda_1\ge\ldots\ge\lambda_r>0$.
    Let $\lambda+k$ be the set of all partitions $\mu$ of $n$ with entries $\mu_1\ge\ldots\ge\mu_{r}\ge \mu_{r+1}\ge 0$ such that for all $i$ $\mu_i\ge \lambda_i$ and $\lambda_i\ge \mu_{i+1}$.
\begin{theorem}[Pieri's formula, (\emph{cf}. \cite{macdonald1995})]
    Assume $k\le n\in \NN, \,n< \ell$ and let $\lambda=(\lambda_1,\ldots,\lambda_m)$ be a partition of $n-k$.
    Then \[\id_{k,n-k}(\fo_k\otimes D_\lambda)=\bigoplus_{\mu\in \lambda+k}D_\mu.\]
\end{theorem}
Secondly, we recall that if $n<\ell$ we have
\[\fl[S_n]=\bigoplus_{\lambda\in \pa_n} D_\lambda \otimeslr  D_\lambda.\] More generally, we have that for any  $\lambda,\mu\in\pa_n^\ell$
\[\ho_\fl( D_\lambda \otimeslr  D_\mu,\fl[S_n])=\begin{cases}
    1&\mu=\lambda,\\
    0&\text{otherwise.}
\end{cases}\]
These two ingredients allow us to compute explicitly for any two partitions $\lambda$ and $\mu$ of $n,m< \ell$ the quantities $d_{\lambda,\mu}^r$ and $d_{\lambda,\mu}$. At the moment we do not know whether there exists a more efficient way of computing these coefficients than pure brute-force.
In \Cref{S:move}, we will need the following definitions.

For $*\in\{L,R\}$, we let $\md_*^n$ be the category consisting of left respectively right submodules of $\fl[S_n]$ which are generated by one element. A morphism in this category is a morphism of modules preserving the inclusion.
In particular the module $\fl[S_n]$ is contained in both $\md_L^n$ and $\md_R^n$.
Moreover the image of a map between such two modules is again contained in $\md_*^n$. For $\alpha$ a composition of $n$, we can define analogously the category $\md_*^\alpha$ consisting of submodules of $\fl[S_\alpha]$ with a generator and define the induction functor 
\[\id_\alpha\colon \md_*^\alpha\ra\md_*^n.\]
Note that every simple (left or right) module $M$ of $\fl[S_n]$ can be embedded in $\fl[S_n]$, hence this embedding gives rise to a corresponding object in $\md_L^n$ and $\md_R^n$.
\subsection{Representation theory of $G_n$}\label{S:reptheor}
We start with some generalities on the structure of $G_n$. For $\alpha=(\alpha_1,\ldots,\alpha_k)$ a composition of $n$, we let $P_\alpha$ be the parabolic subgroup containing the upper-diagonal matrices with Levi-subgroup $M_\alpha\cong G_\alpha\coloneq G_{\alpha_1}\times\ldots\times G_{\alpha_k}$ and unipotent part $N_\alpha$. If $P$ is a parabolic subgroup of $G_n$, we denote by $\overline{P}$ the opposite parabolic subgroup of $G_n$ and we note that $\overline{P_\alpha}$ is conjugate to $P_{\ol{\alpha}}$.
Assume $\alpha$ is of the form \[\alpha=\overbrace{(d,\ldots,d)}^r.\] We then note that the $P_\alpha\times P_\alpha$-orbits on $G_n$ are 
parametrized by $S_r$ and for an element $w\in S_r$ we fix a representative $w_d$ of the corresponding orbit, which is the permutation matrix of $w$, where each $1$ has been replaced with the identity matrix of $G_d$.

Let $G\subseteq G_n$ be a closed subgroup. We denote by $\rep(G)$ the category of smooth, $\fl$-representations of finite length of $G$. The trivial representation is denoted by $\fo_G$ and $\fo_n$ if $G=G_n$. Note that the trivial representation of $S_n$ is denoted by the same letter, but we hope that it is clear from the context which one we are referring to.
If $\alpha=(\alpha_1,\ldots,\alpha_k)$ is a composition of $n$ we write $\rep_\alpha=\rep(G_\alpha)$, $\irr_\alpha$ for the set of isomorphism classes of irreducible representations in $\rep_\alpha$ and set \[\rep\coloneq\bigcup_{n\in \NN}\rep_n,\, \irr\coloneq\bigcup_{n\in \NN}\irr_n.\]
If $H\subseteq G\subseteq G_n$ are closed subgroups, we denote the functor of normalized, compactly supported induction by $\id_H^G$. 
We recall the normalized Jacquet functor and parabolic induction corresponding to parabolic subgroups $P$ of $G_n$ with Levi-subgroup $M$,
which give rise to the exact functors
\[r_P \colon \rep_n\ra \rep(M),\, \id_{P}^{G_n}\colon \rep(M)\ra \rep_n.\]
We write $r_\alpha\coloneq r_{P_\alpha}$ and $\ol{r_\alpha}\coloneq r_{\ol{P_\alpha}}$. 
Recall that by Frobenius reciprocity respectively Bernstein reciprocity $r_{\alpha}$ respectively $\ol{r_\alpha}$ is the left adjoint respectively right adjoint of $\id_{P_\alpha}^{G_n}$ and $(r_{\alpha}(-))^\lor=\ol{r_{\alpha}}((-)^\lor)$. By abuse of notation we will also notate the maps they induce between the respective Grothendieck groups by the same letters.
As is convention, we will write
\[\pi_1\times\ldots\times\pi_k\coloneq \id_{P_\alpha}^{G_n}(\pi_1\otimes\ldots\otimes\pi_k),\, \pi_1\oltimes\ldots\oltimes\pi_k\coloneq \id_{\overline{P_\alpha}}^{G_n}(\pi_1\otimes\ldots\otimes\pi_k).\]
For $\chi$ a smooth $\fl$-character of $\Ff^\times$, we denote by
\[\chi_n\coloneq\chi\circ \det_n\colon G_n\ra\fl\]
and set $\nu_n\coloneq \lvert-\rvert_n.$ 
If $\chi$ is a smooth character of $\Ff^\times$ and $\pi\in \rep_n$, we denote by $\chi\pi=\chi_n\otimes\pi$.
We set for $\pi\in \rep_n$, the degree $\deg(\pi)=n$ and denote injective respectively surjective maps in $\rep_n$ by $\hra$ and $\sra$.
Furthermore, we consider the involution $g\mapsto {}^tg^{-1}$ and denote the twist of a representation $\pi$ by this involution by $\pi\cc$. If $\ell\neq 2$ and $\pi$ is irreducible, then $\pi\cc\cong \pi^\lor$, see \cite{BerZel76}, \cite{Vig96}. As a consequence, we obtain for irreducible representations $\pi,\pi_1,\ldots,\pi_k\in\irr$ an equality
\[\dim_\fl\ho(\pi,\pi_1\times\ldots\times\pi_k)=\dim_\fl\ho(\pi_k\times\ldots\times\pi_1,\pi)\] if $\ell\neq 2$.

If $\pi\in \rep_n$, we denote the corresponding element in the Grothendieck group of $\rep_n$ by $[\pi]$.
Finally, we recall that parabolic induction on $G_n$ is commutative on the level of Grothendieck groups, in particular for $\pi,\pi'\in \irr$ such that $\pi\times\pi'$ is irreducible, $\pi\times\pi'\cong\pi'\times \pi$, see \cite[1.16]{Vig96} for $\ell>2$ and \cite[Proposition 2.6]{MinSec14} in general.
If $n\in \NN,\pi\in \rep$ we write
\[\pi^{\times n}=\overbrace{\pi\times\dots\times\pi}^n.\]
\subsection{Geometric Lemma of Bernstein and Zelevinsky}
Let $P=MU$ and $Q=NV$ be two parabolic subgroups of $G_n$. We recall the description of the functor $r_{Q}\circ \id_P^G\colon \rep(M)\ra \rep(N)$.  Define for $\mathfrak{w}\in P\bs G_n/Q$ the groups
 \[M'\coloneq M\cap w^{-1}Nw,\, N'\coloneq wM'w^{-1},\,V'\coloneq M\cap w^{-1}Vw,\, U'\coloneq N\cap w Uw^{-1},\]
where $w$ is a representative of $\mathfrak{w}$. Let \[\delta_1\coloneq \delta_{U}^{\frac{1}{2}}\cdot\delta_{U\cap w^{-1}Qw}^{-{\frac{1}{2}}},\, \delta_2\coloneq\delta_{V}^{\frac{1}{2}}\cdot\delta_{V\cap wPw^{-1}}^{-{\frac{1}{2}}}\] be characters of $M'$ respectively $N'$. Finally, let $\mathrm{Ad}(w)\colon \rep(M')\ra \rep(N')$ be the pullback by conjugation by $w$ and set $\delta\coloneq \delta_1w^{-1}(\delta_2)$.
Order the $Q$-orbits of $P\bs G_n$ as  $\mathfrak{w}_1,\ldots, \mathfrak{w}_l$ such that $\overline{{\mathfrak{w}_i}}=\bigcup_{i\le j}{\mathfrak{w}_j}.$

 Define for $\mathfrak{w}\in P\bs G_n/Q$ the functor \[F(\mathfrak{w})\coloneq \id_{N'U'}^{N}\circ \mathrm{Ad}(w)\circ \delta\circ r_{V'}\colon \rep(M)\ra\rep(N).\]
Then the following holds.
\begin{lemma}[\cite{BerZel77} Lemma 2.11]
The functor \[F\coloneq r_Q\circ \id_P^G\colon \rep(M)\ra\rep(N)\] has a filtration $0=F_0\subseteq F_1\subseteq F_2\subseteq\ldots\subseteq F_l=F$ with subquotients $F_{i-1}\bs F_i\cong F(\mathfrak{w}_i)$. 
\end{lemma}
\subsection{Cuspidal representations}
Following \cite{MinSec14}, we now give a short treatment of the $\fl$-representations of $G_n$.

A representation $\rho\in\irr_n$ is called cuspidal if for all non-trivial compositions $\alpha$ of $n$, $r_\alpha(\rho)=0$. It is called supercuspidal if there exists no non-trivial composition $\alpha$ and $\pi\in \irr_\alpha$ such that $\rho$ appears as a subquotient of $\id_{P_\alpha}^{G_n}\pi$. We denote the subset of $\irr_n$ consisting of cuspidal respectively supercuspidal representations by $\cus_n$ respectively $\scu_n$ and define \[\cus\coloneq\bigcup_{n\in \NN}\cus_n.\]
We recall the following notions, see for example \cite[§3.4, §4.5]{MinSecSte14}.
Let $\rho\in \cus$ and recall that $\rho\times\chi\rho$ is reducible if and only if $\chi\cong \as^{\pm}$.
We recall the cuspidal line \[\ZZ[\rho]\coloneq \{[\rho \as^k]:\, k\in \ZZ\}\] and denote the cardinality of $\ZZ[\rho]$ by $o(\rho)$. Note that $o(\rho)$ is finite and if $\rho=\fo_1$, then $o(\rho)=d$. Set \[e(\rho)\coloneq \begin{cases}
    o(\rho)&\text{ if }o(\rho)>1,\\
    \ell&\text{ otherwise.}
\end{cases}\]
We can describe the behavior of cuspidal representations with respect to parabolic induction in the simplest case as follows, see for example in \cite[§9]{MinSec14}.
\begin{lemma}\label{L:cus}
    Let $\rho\in \cus$.
    \begin{enumerate}
        \item If $o(\rho)>2$, then
    $\rho\times\rho\as$ admits a unique subrepresentation, a unique quotient, and is of length $2$. Moreover, the quotient and the subrepresentation are not isomorphic.
    \item If $o(\rho)=2$, then
    $\rho\times\rho\as$ admits a unique subrepresentation, a unique quotient, and is of length $3$. Moreover, the quotient and the subrepresentation are not isomorphic. The third irreducible subquotient is a cuspidal representation.
    \item If $o(\rho)=1$ and $\ell\neq 2$, $\rho\times\rho$ is semi-simple of length $2$ and its summands are not isomorphic.
    \item  If $o(\rho)=1$ and $\ell=2$, $\rho\times\rho$ is of length $3$ with a unique subrepresentation and a unique quotient, which are isomorphic. The third irreducible subquotient is a cuspidal representation.
    \end{enumerate}
\end{lemma}
\subsection{Multisegments}
We now recall the combinatorics of multisegments, \emph{cf. }\cite{Zel80}, \cite{MinSec14}.
Let $\rho\in \cus_m$ and $a\le b\in \ZZ$. A segment is a sequence
\[[a,b]_\rho=([\rho\as^a],\ldots,[\rho\as^b])\] and two segments $[a,b]_\rho$ and $[a',b']_{\rho'}$ are equal if and only if $\rho'\as^a\cong \rho\as^{a'}$ and $b-a=b'-a'$. We also let $[a,b]_\rho^\lor=[-b,-a]_{\rho^\lor}$.
Set \[ b_\rho([a,b]_\rho)=a\in 
    \ZZ/(o(\rho)\ZZ),\, e_\rho([a,b]_\rho)=b\in
    \ZZ/(o(\rho)\ZZ).\]
We define length and degree of a segment as $l([a,b]_\rho)\coloneq b-a+1\in \ZZ$ and $\deg([a,b]_\rho)=(b-a+1)m\in \ZZ$. The cuspidal support of $[a,b]_\rho$ is defined as $\mathrm{cusp}([a,b]_\rho)\coloneq[\rho\as^a]+\ldots+[\rho\as^b]$.

A multisegment is a formal finite sum of segments and we extend the length $l$, the degree $\deg$, the notion of cuspidal support, and $(-)^\lor$ linearly. We let $\Ms$ be the set of multisegments and $\Ms(\rho)$ be the set of multisegments consisting only of segments of the form $[a,b]_\rho$. A multisegment is called aperiodic if it does not contain a sub-multisegment of the form
\[[a,b]_\rho+[a+1,b+1]_\rho+\ldots+[a+e(\rho)-1,b+e(\rho)-1]_\rho.\]
We now recall the Zelevinsky classification of {\cite[§9]{MinSec14} and \cite[§6]{Zel80}}. It takes the form of a surjective map
\[\Z\colon \Ms\ra \irr,\]
whose most important properties we recall below.
\begin{theorem}[{\cite[§9]{MinSec14}, \cite[§6]{Zel80}}]\label{T:msZ}
The map
\[\Z\colon \Ms\ra \irr\]
satisfies the following.
\begin{enumerate}
    \item $\deg(\Z(\fm))=\deg(\fm)$.
    \item $\Z$ restricted to aperiodic multisegments is a bijection.
    \item $\Z(\fm)^\lor\cong \Z(\fm^\lor)$.
    \item If $\fm=\fm_1+\fm_2$ then $\Z(\fm)$ is a subquotient of $\Z(\fm_1)\times \Z(\fm_2)$ and appears with multiplicity one in its Jordan-Hölder decomposition.
    \item $\Z([a,b]_\rho)$ is a subrepresentation of $\rho\as^a\times\ldots \times\rho\as^b$.
    \item If $\rho_1,\ldots,\rho_l\in \cus$ lie in pairwise different cuspidal lines and $\fm=\fm_1+\ldots+\fm_k$ with $\fm_i\in \Ms(\rho_i)$,
    \[\Z(\fm)\cong \Z(\fm_1)\times\ldots\times \Z(\fm_k).\]
\end{enumerate}
\end{theorem}
As an example, let us note that if $\rho=\fo_1$, then 
\[\Z([a,b]_\rho)=\abs{b-a+1}{b+a}.\]
Next, let us recall the following properties of representations of the form $\Z(\De)$.
\begin{lemma}[{\cite[Proposition 7.5]{MinSec14}}]\label{L:rese}
    Let $[a,b]_\rho$ be a segment and $(s,t)$ a composition of $\deg([a,b]_\rho)$. If $\deg(\rho)$ does not divide $s$, then
    \[r_{s,t}(\Z([a,b]_\rho))=0.\]
    Otherwise, write $s=u\deg(\rho)$.
    Then \[r_{s,t}(\Z([a,b]_\rho))=\Z([a,a+u-1]_\rho)\otimes \Z([a+u,b]_\rho).\]
\end{lemma}
\begin{lemma}[{\cite[Theorem 6.4.1]{Dro24}}]\label{L:end}
    Let $\fm=\De_1+\ldots+\De_k\in \Ms(\rho)$ be a multisegment and $\Z(\fn)$ an irreducible subquotient of
    \[\Z(\De_1)\times\ldots\times \Z(\De_k).\] If $e_\rho(\De_1)=\ldots=e_\rho(\De_k)$, then for any segment $\Gamma\in \fn$, we have $e_\rho(\Gamma)=e_\rho(\De_1)$.
    Similarly,
    if $b_\rho(\De_1)=\ldots=b_\rho(\De_k)$, then for any segment $\Gamma\in \fn$, we have $b_\rho(\Gamma)=b_\rho(\De_1)$.
\end{lemma}
Let $\rho\in\cus$. We say $\pi\in \irr$ is \emph{strongly $\rho$-reduced} if there exists no $\tau\in\irr$ and integer $c\le 0$ such that $\pi\hra \Z([c,0]_\rho)\times \tau$.

On the other hand, we call a representation $\pi$ \emph{$\rho$-saturated} if there exist integers $c_1,\ldots,c_k\le 0$ such that $\pi\hra\Z([c_1,0]_\rho)\times\ldots\times \Z([c_k,0]_\rho)$.
Note that if $\Z(\fm)$ is $\rho$-saturated, then every segment in $\fm$ ends in $0$ by \Cref{L:end}.
\begin{lemma}\label{L:reduced}
    Let $\rho\in \cus$ and $\pi\in \irr$. There exists a unique $\rho$-saturated representation $\rho(\pi)$ and a $\rho$-reduced representation $\D_\rho(\pi)$ such that $\pi\hra\rho(\pi)\times \D_\rho(\pi)$.
\end{lemma}
\begin{proof}
    Using \Cref{T:msZ} (6) it is easy to show that it suffices to prove the claim if the cuspidal support of $\pi$ consists only of elements in $\ZZ[\rho]$. We can moreover assume $o(\rho)>1$, as otherwise the claim is trivial. In particular $\ell\neq 2$. 
    We first show the existence of such a pair of representations. We can find integers
    $c_1,\ldots,c_k\le 0$ such that \[\pi\hra\Z([c_1,0]_\rho)\times\ldots\times \Z([c_k,0]_\rho)\times \pi'\] and $\pi'$ is strongly $\rho$-reduced. We thus obtain a subquotient $\sigma$ of $\Z([c_1,0]_\rho)\times\ldots\times \Z([c_k,0]_\rho)$ such that 
    \[\pi\hra\sigma\times\pi'.\]
    We need to show that $\sigma$ is not just a subquotient but rather a subrepresentation of the above induced representation to show that it is saturated.
    Since $\ell\neq 2$, we have a non-zero map
    \[\pi'\times\Z([c_1,0]_\rho)\times\ldots\times \Z([c_k,0]_\rho)\sra \pi\hra \sigma\times \pi'\]
    and by Frobenius reciprocity a map
    \[r_{\deg(\sigma),\deg(\pi')}(\pi'\times\Z([c_1,0]_\rho)\times\ldots\times \Z([c_k,0]_\rho))\sra \sigma\otimes \pi'.\]
    We apply the Geometric Lemma and will show that this map cannot vanish on \[\Z([c_k,0]_\rho)\times\ldots\times \Z([c_1,0]_\rho)\otimes \pi',\] which in turn would imply that $\sigma$ is a subrepresentation of $\Z([c_1,0]_\rho)\times\ldots\times \Z([c_k,0]_\rho)$. Indeed, otherwise there would exist a non-trivial composition $(a,b)$ of $\deg(\sigma)$, non-trivial quotient $\pi_1$ of $r_{a,b}(\Z([c_k,0]_\rho)\times\ldots\times \Z([c_1,0]_\rho))$, and an irreducible representation $\tau$ 
    such that there exists a $G_{\deg \pi'}$-invariant morphism
    \[\id_{G_a\times P_{b,\deg\tau}}^{G_b\times G_{\deg\pi'}}(\pi_1\otimes\tau) \sra \pi'.\]
    But by the Geometric Lemma we can give an explicit composition series of \[r_{a,b}(\Z([c_k,0]_\rho)\times\ldots\times \Z([c_1,0]_\rho)),\]
all of whose $G_b$-part are again of the form
\[\Z([c_1',0]_\rho)\times\ldots\times \Z([c_j',0]_\rho).\]
We thus obtain that $\pi'$ is not strongly $\rho$-reduced, a contradiction.

To prove that such $\sigma$ and $\pi'$ are unique, 
we employ the same strategy of analyzing the map obtained by Frobenius reciprocity. Indeed, let $\sigma',\pi''$ be a second possible such pair. We then have a non-zero map \[\pi''\times\sigma'\sra \pi\hra\sigma\times\pi'\]
and hence by Frobenius reciprocity a non-zero map
\[r_{\deg(\sigma),\deg(\pi')}(\pi''\times\sigma')\sra \sigma\otimes\pi'.\]
Again we use the Geometric Lemma and repeat the above argument. The only extra step necessary is that $\pi_1$ in the above notation is a priori just a quotient of $r_{a,b}(\sigma')$. However, since for suitable $c_1',\ldots,c_l'$
\[r_{a,b}(\Z([c_k',0]_\rho)\times\ldots\times \Z([c_1',0]_\rho))\sra r_{a,b}(\sigma'),\] the same argument works.
\end{proof}
We also extract the following from the proof.
\begin{corollary}\label{C:soc}
    Let $\pi$ be strongly $\rho$-reduced and $\sigma$ $\rho$-saturated. Then
    \[\dim_\fl\ho(\pi\times\sigma,\sigma\times\pi)=1\]
    and every morphism factors through the necessarily irreducible unique subrepresentation of $\sigma\times\pi$.
\end{corollary}
\begin{rem}
    The reader familiar with the notion of $\rho$-derivatives might wonder if there is a connection between the above notions and $\rho$-derivatives. Indeed, in the language of \cite{Min08}, $\pi$ being strongly $\rho$-reduced implies that $\pi$ is $\rho$-reduced but not the other way around. However, the goal of the above definition was more to mimic the definition of being reduced respectively saturated with respect to a segment, see \cite[§7]{LapMin25} for a definition. If we are in the banal setting, then a representation $\pi$ is strongly $\rho$-reduced if and only if for all segments $\De=[c,0]_\rho$, $\pi$ is reduced with respect to $\De$ in the language of \cite{LapMin25}.
\end{rem}
\section{Intertwining operators}\label{S:inter}
We recall in this section the relevant notions of $\ell$-modular intertwining operators of \cite[§7]{Dat05}, see also \cite[§3]{Dro24}.
Let $P,Q$ be parabolic subgroups of $G_n$ with the same Levi-component $M$, unipotent parts $U_P$ and $U_Q$, and $\sigma$ a representation of $M$.
 We refer to \cite[§7]{Dat05} and \cite[IV]{Waldspurger2003Plancherel}, see also \cite[§3]{Dro24}, for a more precise treatment of the following definition.
In \cite[§7]{Dat05} the author defines a map 
\[M_{\sigma,Q,P}\colon \id_{Q}^{G_n}(\sigma\otimes\psi)\ra \id_P^{G_n}(\sigma\otimes\psi),\]
where $\psi\colon M\ra \fl(T)$ is a suitable generic character.
This map is uniquely defined by the property that if $f\in \id_{Q}^{G_n}(\sigma\otimes\psi)$ with $\mathrm{supp}(f)\cap QP$ compact mod $Q$, 
\[M_{\sigma,Q,P}(f)(1_n)=\int_{U_{\overline{Q}}\cap U_P}f(u)\,\mathrm{d}u,\]
where the Haar-measure on $U_{\overline{Q}}\cap U_P$ comes from a fixed Haar-measure on $\mathrm{F}$. This choice of Haar-measure will always introduce an ambiguity, since it implies that these intertwiners are defined up to a scalar. We will at the end of this chapter fix a such choice once and for all.
As in \cite[§3]{Dro24}, we find a minimal $\Lambda(\sigma,Q,P)\in \NN$ such that $(T-1)^{\Lambda(\sigma,P',P)}M_{\sigma,Q,P}(f)$ has no pole at $T=1$ for all $f\in \id_{Q}^{G_n}(\sigma\otimes\psi)$ with no pole at $T=1$.
Finally, if $\sigma$ is irreducible, or a subquotient of representations induced from irreducible representations, \[M_{\sigma,\overline{P},P}\circ M_{\sigma,P,\overline{P}}\in \fl[T,T^{-1}]^\times.\] We denote the order of the pole or zero of this scalar by $\alpha(\sigma,P)$. We call $M_{\sigma,Q,P}$ \emph{regular} if $\Lambda(\sigma,Q,P)=0$ and all the intertwining operators that will be crucial to us will turn out to be regular.
Finally, if $P=P_\alpha,\,\alpha=(\alpha_1,\ldots,\alpha_k)$ we make the following definition. For $w\in S_k$, let $P_w$ be the standard parabolic subgroup corresponding to  $w\alpha=(\alpha_{w(1)},\ldots,\alpha_{w(k)}).$ Then there exists a permutation matrix $w'\in G_n$ such that $P_w$ is conjugated to a parabolic subgroup $P_w'$ with Levi-factor $M_\alpha$ and we take $w'$ to be such a permutation of minimal length. We define
\[M_{\sigma,w,P}\coloneq (T-1)^{\Lambda(\sigma,P_w',P)}M_{\sigma,P_w',P}\lvert_{T=1}\circ \ad(w') \colon \id_{P_w}^{G_n}\ad(w')(\sigma)\ra \id_P^{G_n}\sigma\]
and $\Lambda(\sigma,w,P)\coloneq \Lambda(\sigma,P_w',P)\in\NN$. 
If $w$ is the maximal element in $S_k$, \emph{i.e.} $P_w'=\overline{P}$ and $\sigma=\sigma_1\otimes\ldots\otimes \sigma_k$, we write
\[M_{\sigma_1,\ldots,\sigma_k}\coloneq M_{\sigma,w,P},\,\Lambda(\sigma_1,\ldots,\sigma_k)\coloneq \Lambda(\sigma,w,P),\,\alpha(\sigma_1,\ldots,\sigma_k)\coloneq \alpha(\sigma,P).\]
\begin{theorem}\label{T:PInt}
    The above objects satisfy the following properties.
\begin{enumerate}
    \item If $\sigma'$ is a subrepresentation of $\sigma$, then $\Lambda(\sigma,w,P)\ge \Lambda(\sigma',w,P)$ and
    \[\restr{M_{\sigma,w,P}}{\id_{P_w}^{G_n}\ad(w')(\sigma')}=\begin{cases}
        M_{\sigma',w,P}&\text{ if }\Lambda(\sigma,w,P)=\Lambda(\sigma',w,P),\\0&\text{ otherwise.}
    \end{cases}\]
    A similar statement holds for quotients of $\sigma$.
    \item Assume that $\sigma$ is irreducible. Then $\Lambda(\sigma,w,P)$ vanishes if and only if $M_{\sigma,w,P}(f)(1_n)$ does vanish for all $f\in \id_{P_w}^{G_n}\ad(w)(\sigma)$ supported on \[\bigcup_{w'\in \overline{\mathfrak{w}},\, w'\notin \mathfrak{w}}\mathfrak{w}\] and does not vanish on the set of $f\in \id_{P_w}^{G_n}\ad(w)(\sigma)$ supported on
    \[\bigcup_{w'\in \overline{\mathfrak{w}}}\mathfrak{w},\]
    where in both cases $\mathfrak{w}$ ranges over the orbits in $P_w\backslash G/P$.
    \item If $w$ is the identity, $\Lambda(\sigma,w,P)=0$ and $M_{\sigma,w,P}=\id_P^G(1_\sigma)$.
    \item[] For the next 3 points assume that $w$ is the maximal element in $S_k$.
    \item We have that $\alpha(\sigma,P)$ is independent of $P$ and
    \[\fd(\sigma,w,P)\coloneq \Lambda(\sigma,w,P)+\Lambda(\mathrm{Ad}(w')(\sigma),w,\mathrm{Ad}(w')(\overline{P}))+\alpha(\sigma,P)\ge 0\]
    \item If $\sigma$ is an irreducible representation of $M$, 
    \[M_{\mathrm{Ad}(w')(\sigma),w,\mathrm{Ad}(w')(\overline{P})}\circ M_{\sigma,w,P}\] is a scalar. It is non-zero if and only if $\fd(\sigma,w,P)=0$.
    \item If $\sigma'$ is a subquotient of $\sigma$, then $\alpha(\sigma,P)=\alpha(\sigma',P)$.
    \item Let $\sigma=\sigma_1\otimes\sigma_2\otimes\sigma_3$ be irreducible and $P=P_{n_1,n_2,n_3}$. Then
    \[\alpha(\sigma_1\times\sigma_2\otimes\sigma_3,P_{n_1+n_2,n_3})=\alpha(\sigma_1,\sigma_3,P_{n_1,n_3})+\alpha(\sigma_2,\sigma_3,P_{n_2,n_3}).\]
\item[] From now on we assume that $P=P_{d,\ldots,d}$, $\sigma=\sigma'\otimes\ldots\otimes\sigma'$ is irreducible, and note that in this case $w'=w_d$ from \Cref{S:reptheor}.
    \item Let $v\in S_k$ and assume that for all $w$ of length $1$ we have that $M_{\sigma,w,P}$ is an isomorphism and $\Lambda(\sigma,w,P)=0$. Then
    $M_{\sigma,w,P}\circ M_{\sigma,v,P}=M_{\sigma,vw,P}$ and $\Lambda(\sigma,w,P)=0$ for all $w$.
    \item Assume that $\len(w)=1$ and set $Q=P\cup P_w'$ with Levi-subgroup $M$. Then
    $M_{\sigma,w,P}=\id_Q^{G_n}(M_{\sigma,w,M\cap P})$.
\end{enumerate}
\end{theorem}
\begin{proof}
    All properties except (8) can be found in \cite[§7]{Dat05}, see also \cite[§3]{Dro24} or \cite[§4]{Dro25}.
    For (8), we proceed as follows by induction on $\len(w)$.
    If $n=2$, the claim follows from (6) and (9), acting as our base case. Moreover, it is easy to see that it is enough to prove the claim for $l(v)=1$. We argue by induction on $l(w)$. If $l(vw)< l(w)$, we know by the induction hypothesis already that $M_{\sigma,vw,P}\circ M_{\sigma,v,P}=M_{\sigma,w,P}$ and hence the claim follows from (9) and the assumption on $M_{\sigma,v,P}$.
     It thus suffices to prove the claim in the case $l(vw)=l(v)+1$. But this is shown in \cite[Proposition 7.8]{Dat05}.
\end{proof}
\begin{lemma}\label{L:int1}
    Let $\rho\in \cus$ and assume that $o(\rho)=1$. Then $M_{\rho,\rho}$ is an isomorphism and not a scalar.
\end{lemma}
\begin{proof}
Assume $M_{\rho,\rho}$ is a scalar which is by \Cref{T:PInt}(2) equivalent to $\Lambda(\rho,\rho)>0$. Moreover, by \Cref{T:PInt}(5) this implies $\fd(\rho,\rho)=0$ which in turn implies $\alpha(\rho,\rho)<0$. But by \cite[1.6]{Sil80}, see also \cite[§ 8.3]{Dat05}, $\alpha(\rho,\rho)=0$. Thus the map is not a scalar.
We also just showed that $\Lambda(\rho,\rho)=0$, implying that $\fd(\rho,\rho)=0$ and hence $M_{\rho,\rho}$ is an isomorphism.
\end{proof}
\begin{lemma}\label{L:basecase}
    Let $\rho,\rho'\in \cus$ with $o(\rho)>1$. Then 
    \[\Lambda(\rho,\rho')=\begin{cases}
        1&\text{ if }\rho'\cong \rho\\ 0&\text{ otherwise,}
    \end{cases}\]
    and
            \[\alpha(\rho,\rho')=\begin{cases}
        1&\text{ if }o(\rho)>2, \,\rho'\cong \rho\lvert-\lvert,\\ 2&\text{ if }o(\rho)=2,\, \rho'\cong \rho\lvert-\lvert,\\
        -2&\text{ if } \rho'\cong \rho,\\
        0& \text{ otherwise.}
    \end{cases}\]
\end{lemma}
\begin{proof}
Regarding $\alpha$, the claim follows again from \cite[1.6]{Sil80}, see also \cite[8.4]{Dat05}.
Since $\rho\times \rho$ is irreducible in the case $o(\rho)>1$, it follows that $M_{\rho,\rho}$ has to be an isomorphism. Thus $\fd(\rho,\rho)=0$ and hence $\Lambda(\rho,\rho)=1$.
If $\rho'\ncong \rho$, $\Lambda(\rho,\rho')=0$ by \Cref{T:PInt}(2).
\end{proof}
\begin{lemma}\label{L:int2}
    Let $\De=[1,o(\rho)]_\rho$.
    Then $M_{\Z(\De),\Z(\De)}$ is a morphism different from a scalar and $\Lambda(\Z(\De),\Z(\De))=0$.
\end{lemma}
\begin{proof}
The case $o(\rho)=1$ is \Cref{L:int1}, hence we can assume that $o(\rho)>1$.
    Using the above and \Cref{T:PInt} we can show that $\alpha(\Z(\De),\Z(\De))=0$. Thus $\fd(\Z(\De),\Z(\De))$ is $0$ if and only if $\Lambda(\Z(\De),\Z(\De))$ is $0$ if and only if $M_{\Z(\De),\Z(\De)}$ is an isomorphism.
    Moreover, by \Cref{T:PInt}(2), any of the above points would imply that $M_{\Z(\De),\Z(\De)}$ is not a scalar.
  
   Finally, if $\Lambda(\Z(\De),\Z(\De))>0$, we have that the map obtained from $M_{\Z(\De),\Z(\De)}$ by Frobenius reciprocity, $M'\colon r_{\deg(\De),\deg(\De)}(\Z(\De)^{\times 2})\ra \Z(\De)\otimes\Z(\De)$ vanishes on the subrepresentation $\Z(\De)\otimes\Z(\De)$ coming from the Geometric Lemma. But the Geometric Lemma also shows that there exists only one other subquotient admitting a map to $\Z(\De)\otimes \Z(\De)$, namely the quotient
   \[r_{\deg(\De),\deg(\De)}(\Z(\De)^{\times 2})\sra \Z(\De)\otimes\Z(\De)\] coming from Frobenius reciprocity. But this would imply that $M_{\Z(\De),\Z(\De)}$ is a scalar, something that we already excluded.
\end{proof}
\begin{corollary}\label{L:directsum}
    Assume $\ell\neq 2$. Then
    \[\Z(\De)\times\Z(\De)=\Z([1,2o(\rho)]_\rho)\oplus \Z(\De+\De).\]
\end{corollary}
\begin{proof}
   Since $M_{\Z(\De),\Z(\De)}$ is a non-zero morphism squaring to a scalar and $\ell\neq 2$, $\Z(\De)\times \Z(\De)=\pi_1\oplus\pi_2$ with $\pi_i\neq 0$. Moreover, we saw during the above proof that $\ed(\Z(\De),\Z(\De))$ is two dimensional and hence $\ho(\pi_1,\pi_2)= 0$. We know by \cite[Theorem 6.4.1]{Dro24} that the only irreducible subquotients of $\Z(\De)\times \Z(\De)$ are of the form $\Z([1,2o(\rho)]_\rho)$ or $\Z(\De+\De)$, the latter of which appears with multiplicity one. It is therefore easy to see that $\pi_1$ and $\pi_2$ have to be irreducible and of the desired form.
\end{proof}
Let $\De=[1,o(\rho)]_\rho$.
For the rest of the paper, we normalize $M_{\Z(\De),\Z(\De)}$ as follows. In the case $\ell\neq 2$, $\Z(\De)\times \Z(\De)=\Z([1,2o(\rho)]_\rho)\oplus \Z(\De+\De)$ is semi-simple. We demand that $M_{\Z(\De),\Z(\De)}^{2}=\fo_{\Z(\De)\times \Z(\De)}$ and that the kernel of $\fo_{\Z(\De)\times \Z(\De)}-M_{\Z(\De),\Z(\De)}$ equals $\Z([1,2o(\rho)]_\rho)$.

If $\ell=2$, and hence $o(\rho)=1$, we saw that $M_{\rho,\rho}$ is an isomorphism, which is not a scalar.
We demand that $M_{\rho,\rho}^{ 2}=1$, and hence $M_{\rho,\rho}-\fo_{\rho\times\rho}$ is a non-zero morphism, which squares to $0$. We note at this point that the image of $M_{\rho,\rho}-\fo_{\rho\times\rho}$ has to be $\Z([0,1]_\rho)$ by \Cref{L:cus}.
\section{Representations and their Endomorphisms}\label{S:move}
In this section we study representations of the form 
\[\Z(\De)^{\times n},\, \De=[1,o(\rho)]_\rho,\, n\in \mathbb{N}.\]
For the rest of the section we fix a cuspidal representation $\rho$ with $\rho^\lor\cong \rho\lvert-\lvert$. Moreover, for $\pi\in\rep$, we will consider $\ed(\pi)$ not just as a vector-space, but rather as an algebra with multiplication given by composition.

To ease notation, we will write $d=o(\rho)$ and $[a,b]_\rho=[a,b]$. Moreover, we fix $\De\coloneq [1,d]$ and set $\Pi_n\coloneq \Z(\De)^{\times n}$. Note that $\Pi_n$ is self-dual, since $\Z(\De)^\lor\cong \Z(\De)$.
By \Cref{T:PInt}(8) and (9) together with \Cref{L:int1}, \Cref{L:int2}, we can define the map
    \[T\colon \fl[S_n]\ra \ed(\Pi_n)\] sending $w\in S_n\mapsto T(w)\coloneq M_{\Z(\De)^{\otimes n}, w,P_{d,\ldots,d}}$. Note that $M_{\Z(\De)^{\otimes n}, w,P_{d,\ldots,d}}$ is only defined up to a non-zero scalar, which we pin down by choosing for a simple transposition $s$ the corresponding intertwining operator as in the end of the last section.
\begin{corollary}\label{C:important}
The map
    \[T\colon\fl[S_n]\ra \ed(\Pi_n)\]
    is an isomorphism of algebras.
\end{corollary}
\begin{proof}
     By \Cref{T:PInt}(2) the $T(w)$'s are seen to be linearly independent via Frobenius reciprocity. We conclude that $T$ is injective.
    Moreover, it is easy to see from Frobenius reciprocity and the Geometric Lemma that 
    \[\dim_\fl \ed(\Pi_n)\le n!.\]
    Thus the above map is an isomorphism.
\end{proof}
\begin{lemma}\label{C:dual}
    The natural isomorphism $(-)^\lor\colon \ed(\Pi_n)\ra \ed(\Pi_n^\lor)\cong \ed({\Pi_n})$ corresponds under the above morphism $T$ to the map $w\mapsto w^{-1}$.
\end{lemma}
\begin{proof}
    It suffices to check for any $w\in S_n$ that \[M_{\Z(\De)^{\otimes n}, w,P_{d,\ldots,d}}^\lor=M_{\Z(\De)^{\otimes n}, w^{-1},P_{d,\ldots,d}}.\] 
    Let $s_{i_1}\cdots s_{i_k}$ be a reduced expression of $w$.
    By \Cref{T:PInt}(8), we can write
    \[M_{\Z(\De)^{\otimes n}, w,P_{d,\ldots,d}}^\lor=M_{\Z(\De)^{\otimes n}, {s_{i_k}},P_{d,\ldots,d}}^\lor\circ\ldots \circ M_{\Z(\De)^{\otimes n}, {s_{i_1}},P_{d,\ldots,d}}^\lor \]
    and 
    \[M_{\Z(\De)^{\otimes n}, w^{-1},P_{d,\ldots,d}}=M_{\Z(\De)^{\otimes n}, {s_{i_k}},P_{d,\ldots,d}}\circ\ldots \circ M_{\Z(\De)^{\otimes n}, {s_{i_1}},P_{d,\ldots,d}}.\]
    It thus suffices to prove the claim for simple reflections $s_i$.
    By \Cref{T:PInt}(9) it suffices to prove the claim for $n=2$ and $s$ the unique reflection in $S_2$.
    The claim then follows from the discussion after \Cref{L:directsum}.
\end{proof}
Let us make the following definition. For $\pi$ a subquotient of $\Pi_n$, we note that by considering central characters and the Geometric Lemma, $r_{d\alpha}(\pi)=\pi_1\oplus\pi_2$, where all irreducible subquotients of $\pi_1$ are subquotients of $\Pi_{\alpha_1}\otimes\ldots\otimes \Pi_{\alpha_k}$ and $\pi_2$ does not contain any such subquotients. We denote 
\[\rs_\alpha(\pi)\coloneq \pi_1.\]
Let $\alpha=(1,\ldots,1)$. As a corollary of the proof of \Cref{C:important}, we obtain the following.
\begin{corollary}
    Let $\pi$ be a subquotient of $\Pi_n$. Then $\rs_{1,\ldots,1}(\pi)$ is semi-simple.
\end{corollary}
\begin{proof}
    We saw in the proof of \Cref{C:important} that $\rs_{1,\ldots,1}(\Pi_n)$ is semi-simple of length $n!$. By exactness of the Jacquet functor, the same is true for any subquotient.
\end{proof}
\begin{corollary}\label{L:quot}
    Let $\pi$ be a subrepresentation of $\Pi_n$ and $\tau$ the corresponding quotient. Then we have short exact sequences
    \[0\ra \ho(\tau,\Pi_n)\ra \ho(\Pi_n,\Pi_n)\ra \ho(\pi,\Pi_n)\ra 0,\]
    \[0\ra \ho(\Pi_n,\pi)\ra \ho(\Pi_n,\Pi_n)\ra \ho(\Pi_n,\tau)\ra 0.\]
\end{corollary}
\begin{proof}
    We apply Frobenius reciprocity and use that $\rs_{1,\ldots,1}(\pi), \rs_{1,\ldots,1}(\tau)$ and $\rs_{1,\ldots,1}
(\Pi_n)$ are semi-simple.
\end{proof}
Let us denote by $\md_L(\Pi_n)$ the category whose objects are pairs $(\pi,i)$, where $\pi$ is a quotient of $\Pi_n$ and $i\colon \pi\hra \Pi_n$ is an embedding. Similarly, we denote by $\md_R(\Pi_n)$ the category whose objects are pairs $(\pi,p)$, where $\pi$ is a subrepresentation of $\Pi_n$ and $p\colon \Pi_n\ra \pi$ is a surjection.
In both cases we take morphisms to be morphisms between the underlying representations preserving the inclusion respectively surjection.
Note that any irreducible subrepresentation or quotient of $\Pi_n$ gives rise to an element in $\md_L(\Pi_n)$ and $\md_R(\Pi_n)$. Indeed, since $\Pi_n^\lor\cong \Pi_n$ and any subquotient of $\Pi_n$ is self-dual by \Cref{T:msZ} and \cite[Theorem 6.4.1]{Dro25}, the claim follows.

For $(\pi,i)\in \md_L(\Pi_n)$ we denote by $(\pi,i)^\lor$ the dual representation in $\md_R(\Pi_n)$ and vice versa.

Recall now the categories $\md_*^n$ of \Cref{S:symmetric}.
We then define functors for $*\in \{L,R\}$ via
    \[\im_*\colon \md_*^n\ra \md_*(\Pi_n),\, M\mapsto \mathrm{Im}(f).\] Here $f$ is the generator of $M$ and we consider $f$ via the isomorphism $T$ as an element in $\ed(\Pi_n)$. If $*=L$ We then denote by $\im(f)$ the subspace of $\Pi_n$ given by the image of $f$ under the above isomorphism, together with its natural inclusion into $\Pi_n$. On the other hand, if $*=R$, we consider $\im(f)$ as a quotient of $\Pi_n$, together with the natural projection.
    \[\ho_*\colon \md_*(\Pi_n)\mapsto \md_*^n,\, \pi\mapsto \begin{cases}
        i(\ho(\Pi_n,\pi))&\text{ if }*=L,\\
                p^*(\ho(\pi,\Pi_n))&\text{ if }*=R.
    \end{cases}\]
Replacing $n$ with a composition $\alpha$ gives rise to analogous functors.    
\begin{prop}\label{P:movimen}
    Let $*\in \{L,R\}$. Then $\im_*$ and $\ho_*$ are well defined, mutually inverse, $\id_\alpha(\im_*)=\ho_*(\id_{r\alpha})$. Moreover, $\im_L(\widetilde{M})=\im_R(M)^\lor$ and vice versa for $\ho_*$.
\end{prop}
\begin{proof}
    We only treat the case $*=L$, the other one follows similarly. Since $\pi$ is a quotient of $\Pi_n$, any surjection $\Pi_n\sra \pi$ gives a generator of $i(\ho(\Pi_n,\pi))$ by \Cref{L:quot}
    This also implies that \[\im_L(\ho_L(\pi,i))=(\pi,i).\]
    The claim that $\ho_L(\im_L(M))=M$ follows immediately from \Cref{L:quot}.
    Finally, the compatibility with induction is also obvious and the compatibility with taking duals follows from \Cref{C:dual}.
\end{proof}
We denote by $\Z(\lambda)\coloneq \im_L(D_\lambda)$. 

Let $\lambda=(\lambda_1,\ldots,\lambda_k)$ be a partition of $n$. We associate to $\lambda$ a multisegment $\fm_\rho(\lambda)\coloneq [1,\lambda_1d]_\rho+\ldots+[1,\lambda_kd]_\rho$ and note that by \cite[Theorem 6.4.1]{Dro25} every irreducible subquotient of $\Pi_n$ is of the form $\Z(\fm(\lambda))$ for some partition of $n$.
\begin{lemma}\label{L:classmatch}
    Let $\lambda$ be an $\ell$-restricted partition of $n$. Then
    \[\Z(\lambda)\cong \Z(\fm_\rho(\lambda)).\]
\end{lemma}
\begin{proof}
    We start with $\lambda=(n)$. We start by showing that the unique inclusion $\Z(\fm(n))\hra \Pi_n$ is preserved after composing with $1_{\Z(\De)}\times\ldots\times M_{\Z(\De),\Z(\De)}\times \ldots \times 1_{Z(\De)}$. By the normalization we chose, \[1_{\Z(\De)}\times\ldots\times M_{Z(\De),\Z(\De)}\times \ldots \times 1_{Z(\De)}\] acts trivially on $\Z(\De)\times \ldots\times \Z([1,2d]_\rho)\times \ldots\times \Z(\De)$ and $\Z([1,nd])$ lies in this subspace by \Cref{L:rese} and Frobenius reciprocity.
    Now by definition $\hom_L(\Z(\fm(n))\hra \Pi_n)$ is thus by \Cref{P:movimen} associated to the object of $\md_L^n$ invariant under all elementary transpositions and hence has to correspond to the trivial representation with its unique inclusion. 
    By \Cref{P:movimen} the claim then follows for $\lambda=(n)$.
    We now assume we have proven the claim for all $\mu\le \lambda$. It suffices now to show that $\Z(\lambda)$ is a subquotient of $\im_L(M_\lambda)$. It would follow that $\Z(\lambda)=\Z(\fm(\mu))$ for $\mu\le \lambda$. If $\mu\neq \lambda$, it would follow that $D_\mu\cong D_\lambda$, a contradiction. Since $\im_L(Sp_\lambda)$ is contained in $\im_L(M_\lambda)$ it suffices to show that $\Z(\lambda)$ appears in $\im_L(Sp_\lambda)$. However, this follows again by \Cref{P:movimen}, since $D_\lambda$ is a quotient of $Sp_\lambda$ and hence we have
    \[\Z(\lambda)^\lor\hra\im_L(Sp_\lambda)^\lor=\im_R(Sp_\lambda^*).\] Thus the claim follows.
\end{proof}
\begin{rem}
    The above can be used to give a new proof that if $o(\rho)=1$, $\Z(\ell\cdot [0,0]_\rho)$ is cuspidal by using that there exists no irreducible representation of $S_\ell$ parametrized $(1,\ldots,1)$. Indeed, if it were not cuspidal, we could embed it into $\Pi_n$ and hence it corresponds under $\mathrm{Hom}_L$ to an element in $\md_L^n$, whose underlying representation is irreducible. But then we can apply \Cref{L:classmatch} to finish the claim.
\end{rem}
Next we consider the case of restriction respectively applying the Jacquet functor.
\begin{lemma}\label{L:rest}
    Let $\ain{k}{1}{n}$. Then \[\rs_{k,n-k}(\Pi_n)\cong \binom{n}{k}\cdot \Pi_k\otimes \Pi_{n-k}.\]
    Moreover, this isomorphism can be chosen such that the image of $T(w)$ in
    \[\ho(\rs_{k,n-k}(Z(\De)^n),\Pi_k\otimes\Pi_{n-k})\cong \restr{\fl[S_n]}{\fl[S_k]\times \fl[S_{n-k}]}\]
     under adjointness corresponds to the element $w$.
\end{lemma}
\begin{proof}
    The Geometric Lemma shows that $\rs_{k,n-k}(\Pi_n)$ has length $\binom{n}{k}$ and all subquotients are isomorphic to the $\Pi_k\otimes \Pi_{n-k}$. It thus suffices to choose $\binom{n}{k}$ linearly independent morphisms
    \[\rs_{k,n-k}(\Pi_n)\sra \Pi_k\otimes \Pi_{n-k}.\]
    However these are given precisely by $T(w)$ where $w$ is a minimal element in one of the $\binom{n}{k}$ $S_{k,n-k}$-cosets in $S_n$, see \Cref{S:symmetric}. These choices also assure that our choices line up in the required sense.
\end{proof}
As a corollary we can prove as \Cref{P:movimen} the following.
\begin{corollary}\label{C:rescom}
    Let $(\pi,p)\in \md_{R}(\Pi_n)$ and $\ain{k}{1}{n}$.
    There exist functorial bijections between the sets
    \[\{(\tau,p'):\,p'\colon \rs_{(k,n-k)}(\pi)\sra \tau, \text{ there exists }i\colon \tau\hra \Pi_k\otimes \Pi_{n-k}\}\]
    and
    \[\{(M,\xi):\,M\hra \restr{\ho_R(\pi,p)}{\fl[S_k]\times \fl[S_{n-k}]},\, \xi\text{ a generator of }M\}\]
    given by $(\tau,p')\mapsto (p'\circ p)(\ho(\tau,\Pi_k\otimes \Pi_{n-k}))$
    and 
    \[(M,\xi)\mapsto \im(f).\]
\end{corollary}
    Again we understand $\im(f)$ to come equipped with the natural surjection or inclusion.
\begin{proof}
    The only non-trivial part is to see that
    $\ho(\rs_{k,n-k}(\pi),\Pi_k\otimes \Pi_{n-k})$ matches up with
    $\restr{\ho_R(\pi,p)}{\fl[S_k]\times \fl[S_{n-k}]}$.
    But this follows immediately the choices made in \Cref{L:rest}.
\end{proof}
Let $\lambda$ be an $\ell$-restricted partition of $m\le n$ and $\pi$ an irreducible subrepresentation of $\Pi_n$. Then we denote by $\rst_\lambda^l(\pi)$ the maximal representation $\Pi$ such that $\rs_{n-m,m}(\pi)\sra \Pi\otimes \Z(\lambda)$ and $\rst_\lambda^r(\pi)$ the maximal representation $\Pi$ such that $\rs_{m,n-m}(\pi)\sra \Z(\lambda)\otimes \Pi$.
\begin{lemma}\label{L:restr}
Let $p,n,m\in \NN$ with $p\le n,m$ and $\pi,\sigma$ two subrepresentations of $\Pi_n$ and $\Pi_m$. Let moreover $\lambda$ be an $\ell$-restricted partition of $n-p$, $\lambda'$ an $\ell$-restricted partition of $m-p$, and $f\colon \rst_\lambda^l(\pi)^\lor\ra\rst_{\lambda'}^r(\sigma)$ a morphism. Then the image of $f$ is of the form $\Pi_1\oplus\ldots\oplus \Pi_k$ with $\Pi_{p}\sra\Pi_i\hra \Pi_{p}$ for all $\ain{i}{1}{k}$.
\end{lemma}
\begin{proof}
    From \Cref{L:rest} it follows that there exist suitable $N\in \NN$ and a surjective map
    \[N\cdot \Pi_{p}\sra \rst_\lambda^l(\pi)\] and similarly for $\rst_{\lambda'}^r(\sigma)$. But then the image of $f$ is the image of a morphism \[N\cdot \Pi_{p}\ra N'\cdot \Pi_{p},\] which is of the desired form.
\end{proof}
Let $\pi_1,\pi_2,\in \irr_n$. Then there is a natural identification of $\ho(\pi_1,\pi_2^\lor)$ with $\ho(S(G_n),\pi_1^\lor\otimes\pi_2^\lor)$ sending a morphism $f$ to the morphism sending a compactly supported function $\phi$ to the functional acting on $v\otimes w$ via
\[\int_{G_n}\phi(g)f(\pi_1(g)v)(w)\, \mathrm{d}g.\]
This identification depends of course on a choice of Haar-measure.
We denote by $\Sigma_n(\pi_1)\subseteq \pi_1^\lor\otimes\pi_1$ the image corresponding to the identity.
\begin{theorem}\label{T:sec4}
    Let $p\le m, n\in \NN$ and let $\Z(\mu),\Z(\mu')$ be two subrepresentations of $\Pi_n$ and $\Pi_m$. Then
    \[\dim_\fl\ho(\id_{P_{d(p,n-p)}\times P_{d(m-p,p)}}(\fo_{d(n-p)}\otimes \fo_{d(m-p)}\otimes S(G_{pd})),\Z(\mu)\otimes \Z(\mu'))=\]\[=\dim_\fl\ho(\id_{P_{d(p,n-p)}\times P_{d(m-p,p)}}(\fo_{d(n-p)}\otimes \fo_{d(m-p)})\otimes \Sigma_{pd}(\Pi_p),\Z(\mu)\otimes \Z(\mu'))=\]\[=\dim_\fl\ho(\fl[\pp_{n,m}^r],D_\mu\otimeslr D_{\mu'}).\]
\end{theorem}
\begin{proof}
    We apply on the right side the second adjointness, the above observation, to obtain that it is equal to
    \[\dim_\fl\ho(\rst_\lambda^l(\Z(\mu)),\rst_{\lambda'}^r(\Z(\mu'))).\] The first equality follows now by \Cref{L:restr}. But then we apply \Cref{C:rescom} to obtain the desired second equality.
\end{proof}
\section{Theta correspondence in type II}\label{S:finalbig}
In this section we study several aspects of the theta correspondence.
\subsection{Geometry of $\mt_{n,m}$}
Let $n,m\in \NN$ and consider the space $\mt_{n,m}$ of $n\times m$ matrices over $\Ff$ on which $G_n\times G_m$ acts by conjugation. For simplicity we will assume $n\le m$. Define for $\ain{r}{0}{n}$ the $G_n\times G_m$ orbits $\mt_{n,m}^r\coloneq \{M\in\mt_{n,m}:\rk(M)=r\}$ and their closures
\[\mt_{n,m}^{\le r}\coloneq \{M\in\mt_{n,m}:\rk(M)\le r\}\] as well as their complement
\[\mt_{n,m}^{\ge r+1}\coloneq \{M\in\mt_{n,m}:\rk(M)\ge r+1\}.\]
We thus obtain the well-known filtration of $S(\mt_{n,m})$ by the rank
\[0\subseteq S(\mt_{n,m}^{\ge n})\subseteq S(\mt_{n,m}^{\ge n-1})\subseteq \ldots\subseteq S(\mt_{n,m}^{\ge 0})=S(\mt_{n,m}),\]
with quotients
\[0\ra S(\mt_{n,m}^{\ge r+1})\ra S(\mt_{n,m}^{\ge r})\ra S(\mt_{n,m}^{r})\ra 0.\]
Next we fix the element \[\eta_r\coloneq \begin{pmatrix}
    1_r&0\\0&0
\end{pmatrix}\]
with stabilizer the inverse image of $\Delta G_r$ under the natural surjection $P_{r,n-r}\times \ol{P_{r,m-r}}\sra G_r\times G_r$.
One thus obtains the well-known isomorphism
\[S(\mt_{n,m}^{\ge r})\cong \sigma_{n,m}^r\coloneq \id_{P_{r,n-r}\times  \overline{P_{r,m-r}}}^{G_n\times G_m}(\xi_{r,n,m}\otimes S(G_r)),\]
where \[\xi_{r,n,m}\coloneq \abs{r}{r-n}\otimes \abs{n-r}{r}\otimes\abs{r}{m-r}\otimes \abs{m-r}{-r}\] is a character on $G_r\times G_{n-r}\times G_{r}\times G_{m-r}$.

Let $\ain{t}{0}{n}$ and consider the open $P_{n-t,t}\times \overline{P_{m-t,t}}$-subset of $\mt_{n,m}$ given by
\[\mt_{n,m,t}\coloneq \left\{\Xi(u,v,m,g):g\in G_t, u\in \mt_{m-t,t},v\in\mt_{t,n-t},m\in\mt_{n-t,m-t}\right\},\]
where \[\Xi(u,v,m,g)\coloneq \begin{pmatrix}
    1&0\\v&1
\end{pmatrix}\begin{pmatrix}
    m&0\\0&g
\end{pmatrix}\begin{pmatrix}
    1&u\\0&1
\end{pmatrix}.\]
    We have 
    \[r_{P_{n-t,t}\times \overline{P_{m-t,t}}}(S(\mt_{n,m,t}))\cong (\abs{n-t}{t}\otimes \abs{n-t}{-t})S(M_{n-t,m-t})\otimes (\abs{t}{t-n}\otimes \abs{t}{m-t})S(G_t)\]
via the map sending the equivalence class of a map $f$ on the left side to 
\[(m,g)\mapsto \int_{\mt_{m-t,t}\times \mt_{t,n-t}}f(\Xi(u,v,m,g))\, \mathrm{d}u \mathrm{d}v\]
for any nontrivial choice of Haar measures.
The following was proved in the complex case, however the exact same arguments work in the modular case.
\begin{theorem}[{\cite[§5.1,5.3]{Dro25II}}]\label{T:reduc3}
    Let $\pi_1,\pi_2,\tau_1,\tau_2\in\irr$ such that $\pi_2$ is $\as^{\frac{n-1}{2}}$-reduced and $\tau_2$ is $\as^{\frac{m-1}{2}}$-reduced.

    If there exists a non-zero morphism
    \[S(\mt_{n,m})\ra \pi_1\times \pi_2\otimes \tau_1^\lor\oltimes\tau_2^\lor,\] then $\deg(\pi_2)=\deg(\tau_2)=t$, $\as^{\frac{m-n}{2}}\pi_2\cong \tau_2$ and the map 
    \[r_{P_{n-t,t}\times \overline{P_{m-t,t}}}(S(\mt_{n,m,t}))\hra  r_{P_{n-t,t}\times \overline{P_{m-t,t}}}(S(\mt_{n,m}))\ra \pi_1\otimes\pi_2\otimes \tau_1^\lor\otimes\tau_2^\lor\] obtained by Frobenius reciprocity is nonzero.
\end{theorem}
\begin{corollary}\label{T:reduc1}
    Let $\pi\in \irr_n,\,\pi'\in\irr_m$ and set $\rho\coloneq \as^{\frac{n-1}{2}},\rho'\coloneq \as^{\frac{m-1}{2}}$. Then
    \[\dim_\fl\ho(S(\mt_{n,m}),\pi\otimes\pi'^\lor)\neq 0\] only if \[\D_{\rho}(\pi)\as^{\frac{m-n}{2}}\cong \D_{\rho'}(\pi').\]
    In this case set $t\coloneq \deg(\D_\rho(\pi))$. We then have an equality
    \[\dim_\fl\ho(S(\mt_{n,m}),\pi\otimes\pi'^\lor)=\dim_\fl\ho(S(\mt_{n-t,m-t}),\as^{-\frac{t}{2}}\rho(\pi)\otimes\as^{\frac{t}{2}}{\rho'}(\pi')^\lor).\]
\end{corollary}
\begin{proof}
    The condition of non-vanishing plus the upper bound of the left side by the right side is an immediate consequence of \Cref{T:reduc3}. We write $\pi_1={\rho}(\pi),\pi_2=\D_{\rho}(\pi), \tau_1={\rho'}(\pi')^\lor,\tau_2=\D_{\rho'}(\pi')^\lor$.
    To show the other direction let $T\colon (\abs{n-t}{t}\otimes \abs{n-t}{-t})S(\mt_{n-t,m-t})\ra \pi_1\otimes\tau_1$ be a non-zero morphism.
    We also fix a non-zero morphism $S\colon (\abs{t}{t-n}\otimes \abs{t}{m-t}S(G_t))\ra\pi_2\otimes\tau_2$, which is unique up to a scalar.
    Then there exists a morphism
    \[\id_P(T)\colon S(M_{n,m})\ra\pi_1\times\pi_2\otimes \tau_1\times \tau_2\] given by
    \[f\mapsto \left((g_1,g_2)\mapsto \int_{\mt_{n-t,t} }T\otimes S\left((m,g)\mapsto f\left(g_1^{-1}\Xi(u,0,m,g)g_2\right)\right)\, \mathrm{d}u\right).\]
We recall the intertwining operator
\[M_{\tau_2,\tau_1}\colon \tau_1\times\tau_2\ra\tau_2\times \tau_1\]
defined in \Cref{S:inter}.
As in the proof of \Cref{L:reduced}, we can show that $\Lambda(\tau_2,\tau_1)=0$. We now let $w$ be a representative of the open $P_{\deg(\tau_2),\deg(\tau_1)}\times P_{t,m-t}$-orbit on $G_n$ as in \Cref{S:inter}. Twisting $\tau_2\times \tau_1$ and $M'$ by conjugating with $w$, we obtain a morphism $M'\colon \tau_1\times\tau_2\ra\tau_1\oltimes\tau_2$, which by a natural analog of \Cref{C:soc} has $\pi'=\soc(\tau_2\times\tau_1)$ as its image.

Next we claim the composition $(1\otimes M')\circ\id_P(T)$ is non-zero. Since $\Lambda(\tau_2,\tau_1)=0$, it suffices by \Cref{T:PInt}(2)
 to show that the image of $\id_P(T)$ contains on the $G_m$-part at least one function whose support is contained in $P_{m-t,t}\ol{P_{m-t,t}}$ and whose integral over the unipotent part of $\ol{P_{m-t,t}}$ does not vanish.
But by construction of $\id_P(T)$, this is the same as computing
\[\int_{\mt_{m-t,t}\times \mt_{t,n-t}}(T\otimes S)\left(f(\Xi(u,v,m,g))\right)\, \mathrm{d}u\, \mathrm{d}v,\]
which can easily be seen not to vanish for suitable $f$. Namely, on the $\mt_{m-t,t}\times \mt_{t,n-t}$-part, one can take a compactly supported function $\phi$ such that the integral over $\mt_{m-t,t}\times \mt_{t,n-t}$ does not vanish and for the $G_t\times G_{m-t}$-part we can find $\phi'$ such that $T\otimes S(\phi')$ does not vanish. Then $f(\Xi(u,v,m,g))\coloneq \phi(u,v)\phi'(m,g)$ has the desired property.
Moreover, we see that the map obtained by Frobenius reciprocity 
\[r_{P_{n-t,t}\times \overline{P_{m-t,t}}}(S(\mt_{n,m,t}))\hra  r_{P_{n-t,t}\times \overline{P_{m-t,t}}}(S(\mt_{n,m}))\ra \pi_1\otimes\pi_2\otimes \tau_1\otimes\tau_2\]
does not vanish and after composing with the isomorphism we mentioned before \Cref{T:reduc3}, recovers the maps $T$ and $S$.
Moreover, the image of $(1\otimes M')\circ\id_P(T)$ is contained in \[\pi_1\times \pi_2\otimes \pi'^\lor\subseteq \pi_1\times\pi_2\otimes\tau_1\oltimes\tau_2.\]
But on the other hand, we can also consider the morphism
    \[\id_{\ol{P}}(T)\colon S(M_{n,m})\ra\pi_1\oltimes\pi_2\otimes \tau_1\oltimes \tau_2\] given by
    \[f\mapsto \left((g_1,g_2)\mapsto \int_{\mt_{t,n-t} }T\otimes S\left((m,g)\mapsto f\left(g_1^{-1}\Xi(0,v,m,g)g_2\right)\right)\, \mathrm{d}v\right).\]
As above, we twist the intertwining operator $M_{\pi_1,\pi_2}$ by a suitable element in the general linear group to obtain a map
\[M\colon\pi_1\oltimes\pi_2\ra\pi_1\times\pi_2.\]
We see analogously to above that 
\[(M\otimes 1)\circ \id_{\ol{P}}(T)\]
is non-zero, the map obtained by Frobenius reciprocity also does not vanish on \[r_{P_{n-t,t}\times \overline{P_{m-t,t}}}(S(\mt_{n,m,t}))\] and moreover agrees with the one obtained from the first construction.
Thus \[(M\otimes 1)\circ \id_{\ol{P}}(T)=(1\otimes M')\circ\id_P(T).\]
Finally, we obtain from \Cref{C:soc} that the image of $(M\otimes 1)\circ \id_{\ol{P}}(T)$ is contained in \[\pi\otimes \tau_1\oltimes\tau_2\subseteq \pi_1\times\pi_2\otimes\tau_1\oltimes\tau_2\] by the same argument. We conclude that it has to have image $\pi\otimes\pi'^\lor$. Finally, we also saw in the course of the construction how to recover $T$ and $S$, thus we obtained the desired lower bound.
\end{proof}
\subsection{Fourier-theoretic aspects}\label{S:fourier}
Recall the Fourier transform
\[\hat{(-)}\colon S(\mt_{n,m})\ra S(\mt_{n,m}).\] In order to define it, we fix an auxiliary smooth additive character $\psi\colon \Ff\ra\fl$ and a Haar-measure $\mathrm{d} x$ on $\mt_{n,m}$. Define then
\[\hat{f}(y)=\int_{\mt_{n,m}}f(x)\psi(\mathrm{tr}(x^ty))\,\mathrm{d} x.\]
By a change of variables, the Fourier-transform gives for any $\pi\in\irr_n,\,\pi'\in\irr_m$ the following equality.
\begin{prop}\label{T:reduc2}
Let $\pi_1\in\irr_n,\pi_2\in\irr_m$. Then
\[\dim_\fl\ho(S(\mt_{n,m}),\pi_1\otimes\pi_2)=\dim_\fl\ho(S(\mt_{n,m}),(\as^{m}\otimes \as^{-n})\pi_1\cc\otimes \pi_2\cc).\]
\end{prop}
As a last piece of the puzzle, we construct the following maps. Recall that $d$ is the order of $q$ in $\fl$.
Assume $n=dn',m=dm'$ and let $r\le n'$.
Let $\alpha$ be the partition $(d,\ldots,d)$ of $dr$.
We consider the closed subspace
\[\mathfrak{P}_r=\left\{\begin{pmatrix}
    u&m\\0&0
\end{pmatrix}:u\in U_{\alpha}, m\in \mt_{dr,m-dr}\right\}.\]
We then define the maps
\[T_r\colon S(\mt_{n,m})\ra\id_{P_{\alpha,n-dr}}^{G_n}(\fo_{M_{\alpha,n-dr}})\otimes \id_{P_{\alpha,m-dr}}^{G_m}(\fo_{M_{\alpha,m-dr}})\]
by \[f\mapsto ((g_1,g_2)\mapsto\int_{\mathfrak{P}_r}f\left(g_1^{-1}\begin{pmatrix}
    u&m\\0&0
\end{pmatrix}g_2\right)\,\mathrm{d}u\,\mathrm{d}m)\]
for some Haar measures $\mathrm{d}u\,\mathrm{d}m$. It follows from the definition that 
$T_r$ vanishes on $S(\mt_{n,m}^{\ge r+1})$ and induces on
$S(\mt_{n,m}^{r})$ a map with image 
\[\id_{P_{n-rd,rd}\times \overline{P_{m-rd,rd}}}^{G_n\times G_m}(\fo_{n-rd}\otimes \fo_{m-rd}\otimes \Sigma_{rd}(\fo_d^{\times r})),\]
see also \Cref{T:sec4}
\subsection{Combining the approaches}\label{S:final}
We define a map
\[\sk\colon \irr_n\times\irr_m\ra \pa^\ell\times \pa^\ell\] uniquely defined by the following properties.
 Let $\pi\in\irr_n,\pi'\in\irr_m$ and $\rho=\as^{\frac{n-1}{2}},\,\rho'=\as^{\frac{m-1}{2}}$, $t=\deg(\D_\rho(\pi))$.
 Recall that in \Cref{S:move} we associate to an $\ell$-restricted partition $\lambda$ a multisegment $\fm_\rho(\lambda)\in\Ms_\rho(\fm)$.
\begin{enumerate}
    \item  If $\pi\cong \Z(\fm_\rho(\lambda))$, $\pi'\cong\Z(\fm_{\rho'}(\mu))^\lor$: \[\sk(\pi,\pi')= \begin{cases}
        (\lambda,\mu)&\text{if }n=m\mod d,\\0&\text{otherwise}.
    \end{cases}\]
    \item If $\D_{\rho}(\pi)\neq 0$ or $\D_{\rho'}(\pi'^\lor)\neq 0$ and $\as^{\frac{m-n}{2}}\D_{\rho'}(\pi)\cong \D_{\rho'}(\pi'^\lor)$: \[\sk(\pi,\pi')= \sk(\as^{-\frac{t}{2}}\rho(\pi),\as^{\frac{t}{2}}\rho'(\pi'^\lor)^\lor).\]
    \item If $\D_{\rho}(\pi)\neq 0$ or $\D_{\rho'}(\pi'^\lor)\neq 0$ and $\as^{\frac{m-n}{2}}\D_{\rho'}(\pi)\ncong \D_{\rho'}(\pi'^\lor)$: \[\sk(\pi,\pi')=\emptyset.\]
    \item If $\ell\neq 2$, $\pi$ is $\rho$-saturated and $\pi'^\lor$ is $\rho'$-saturated: \[\sk(\pi,\pi')= \sk(\as^{m}\pi^\lor, \as^{-n}\pi'^\lor).\]
\end{enumerate}
The above properties define $\sk$ uniquely. Indeed, using first (2) and (3), it suffices to show uniqueness in the case where $\pi$ is $\rho$-saturated and $\pi'^\lor$ is $\rho'$-saturated. Using (4), and then again (2) and (3), one then reduces to the case covered by (1).
\begin{theorem}
    Let $\pi\in\irr_n,\pi'\in \irr_m$. Then
    \[\dim_\fl\ho(S(\mt_{n,m}),\pi\otimes\pi')=d_{\sk(\pi,\pi')}.\]
\end{theorem}
\begin{proof}
    The theorem follows from \Cref{T:reduc1}, \Cref{T:reduc2}, \Cref{T:sec4} and the maps $T_r$ constructed in \Cref{S:fourier}. The only non-trivial part is to see that in the last step the non-vanishing of the Hom-space implies $m=n\mod r$. This can be seen as follows. If the Hom-space is non-zero there exists $r$ with $\sigma_{n,m}^r\ra\pi\otimes \pi'$. This implies, in the language used in the description above, that for every segment $\De\in \fm$ we have $b_{\rho}(\De)=r-n+1$ and also for every segment $\Gamma\in \fn$ we have $e_{\rho'^\lor}(\Gamma)=m-r-1$. But then $m+1-n=r-n+1\mod d$ and $n-1-m=m-r-1\mod d$ and hence $m=n\mod d$.
\end{proof}
    Let us also remark on some properties of $\sk$ that are evident from the above construction.
\begin{enumerate}
    \item Assume that $\pi_1,\pi_2\in\irr_m$ such that the first components of $\sk(\pi,\pi_1),\, \sk(\pi,\pi_2)$ are different from $\emptyset$. Then the first components agree, and we will denote them by $\sk_m(\pi)$. If no $\pi'\in\irr_m$ exists such that the first component of $\sk(\pi,\pi')$ is different fromf $\emptyset$, we set $\sk_m(\pi)=\emptyset$. 
    \item If $m=m'\mod d$, $\sk_m(\pi)=\sk_{m'}(\pi)$.
    \item If $\pi$ is a subrepresentation of 
    $\fo_d^{\times n}$, and hence of the form $\Z(\fm(\lambda))$ for $\lambda\in \pa_n^\ell$, $\sk_m(\pi)=\lambda$ for $m=n\mod d$.
\end{enumerate}
\bibliographystyle{abbrv}
\bibliography{References.bib}
\end{document}